\documentclass[11pt,a4paper]{article}    
\usepackage{amsmath,amsfonts,amssymb,amsthm,amscd}
\usepackage[english]{babel}

\newtheorem{theorem}{Theorem}

\newtheorem{proposition}{Proposition}

\newcommand{\dsp}{\displaystyle}
\newcommand{\eps}{\varepsilon}
\newcommand{\R}{\mathbb{R}}

\numberwithin{equation}{section}

\begin{document}
\title{The gravitational Vlasov-Poisson system with infinite mass and velocities in $\mathbb{R}^3$}
\author{Guido Cavallaro$^*$ and Carlo Marchioro$^+$}
\maketitle
\begin{abstract}
We study existence and uniqueness of the solution to the gravitational Vlasov-Poisson system  
evolving in $\mathbb{R}^3$.
 It is assumed that initially the particles are distributed  according to a spatial density with
 a power-law decay in space, allowing for unbounded mass,  and an exponential decay in velocities given
 by a Maxwell-Boltzmann law. We extend a classical result which holds for systems with finite total mass.
\end{abstract}


\footnotetext{$^*$Dipartimento di Matematica Universit\`a La Sapienza, p.le A. Moro 2,  00185 Roma (Italy),  
cavallar@mat.uniroma1.it}
\footnotetext{$^{+}$International  Research Center M$\&$MOCS   (Mathematics and Mechanics of Complex Systems),
marchior@mat.uniroma1.it}

\section{Introduction and main results}

In the present paper we  deal with the Vlasov-Poisson system with gravitational interaction in the
whole physical space $\mathbb{R}^3$, when the
total mass is infinitely large and the velocities have a Maxwell-Boltzmann distribution.
This model is interesting for astrophysical applications and it has been studied in case of
finite total mass \cite{Ho1, Ho2, R2} or for a suitable perturbation of a given homogeneous background \cite{RR, J}.
It is interesting then to investigate the limit in which the mass of the system increases indefinitely,
in order to show that the properties of the solution are not affected by the size and the total mass of the system.
We focus, in the present paper, on the existence and uniqueness of the solution, which has been obtained in a similar setup, with infinite mass, in case of a plasma \cite{CCM18,  CCM20}, that is for Coulomb interaction. There is a wide literature on the Vlasov-Poisson equation  
\cite{BR, Pf, R, S, W}, see also \cite{G} for a review
of such results,
and \cite{Ca, L, Pa1, Pa2} where it is adopted the so-called method of propagation of moments.
It has  been studied also the case of a plasma with a steady spatial asymptotic behavior \cite{Ch, P, S1, S2, S3},  situations in which the mass of the system
is infinitely large \cite{CCM15,  CCM18, J}, or there is an external magnetic field confining
the system in a given domain \cite{CCM12, CCM, CCM1, CCM15bis, CCM16, CCM17, CCM19}.

The gravitational case, in which the potential energy is negative, can be studied analogously since the absolute value of the potential energy can be controlled by the square root of the kinetic energy, which turns out to be bounded by
the initial data, using conservation of energy. This permits to achieve global existence and uniqueness
in $\mathbb{R}^3$ (see \cite{G} for the possibility of a blow-up in finite time in $\mathbb{R}^d$
for $d\ge 4$). 
It is interesting also to compare what happens in three dimensions for a dust model not described by the Vlasov-Poisson
equation, in which a collapse can occur \cite{R3}. 

This global result  permits to treat the case
with finite total mass, but it is not trivial to be used  for general  initial data
with infinite mass (and energy), in which an accurate limit procedure has to be implemented in order
to achieve the result. This has been done in case of an infinite mass plasma with many species of different charges 
\cite{CCM20}, using the fact that the potential energy is positive, and it deserves to be done for gravitational interaction, since the argument which
permits to perform the limit of infinite mass (via an iterative method) is different.

The strategy is to consider a truncated system by means of a cutoff, for which the total mass is finite and there exists
a unique solution. Then we let the cutoff going to infinity, showing the convergence of the solution
to the one of the infinite mass problem.

We state below the main result of the present paper,  which is discussed in Section 2,
where we introduce a truncated system evolving via a  {\textit{partial  dynamics}.

\bigskip

We denote by $f(x,v,t)$ the mass distribution  at the point of the phase space $(x,v)$ at time $t$.
We let evolve $f$ with the Vlasov-Poisson equation
\begin{equation}
\label{Eq}
\left\{
\begin{aligned}
&\partial_t f(x,v,t) +v\cdot \nabla_x f(x,v,t)  + \, G(x,t) \cdot \nabla_v f(x,v,t)=0  \\
&G(x,t)=-\int_{\R^3 } \frac{x-y}{|x-y|^3} \ \rho(y,t) \, dy     \\
&\rho(x,t)=\int_{\R^3} f(x,v,t) \, dv \\
&f(x,v,0)=f_{0}(x,v)\geq 0,  \qquad  x\in {\R^3 },  \qquad v\in\R^3,  
\end{aligned}
\right.
\end{equation} 
where  $\rho$ is the spatial mass density, $G(x,t)$ is the gravitational field.


Associated to \noindent System (\ref{Eq}) there is the following system for the characteristic equations, along whose solutions $f$ is conserved:
\begin{equation}
\label{ch} 
\begin{cases}
\dsp  \dot{X}(t)= V(t)\\
\dsp  \dot{V}(t)=   G(X(t),t) \\
\dsp (X(0), V(0))=(x,v)  \\
\dsp f(X(t), V(t), t) = f_{0}(x,v).
 \end{cases}
\end{equation}
 We have used here the shortened notation 
\begin{equation}
 \label{2.8}
(X(t),V(t))= (X(t,x,v),V(t,x,v)) 
 \end{equation}
 to represent a characteristic  at time $t$ passing at time $t=0$ through the point $(x,v)$. 
 We have
\begin{equation}
\| f(t)\|_{L^\infty}= \| f_{0} \|_{L^\infty}
\label{f0}
\end{equation}
and also the conservation of the measure of the phase space, by Liouville's theorem.
It is well known that a result of existence and uniqueness of solutions to (\ref{ch}) implies the same result for solutions to (\ref{Eq})
if $f_0$ is smooth.

\noindent  We adopt in the sequel the following notation, we  denote by $C$ 
any positive constant depending only on the initial data and the parameters of the problem,   
reserving to number some of them when relevant.

\noindent The main result of the paper is stated in the following theorem:

\begin{theorem}
Let us fix an arbitrary positive time $T$, and  assume $f_{0}$ satisfies
 the following hypothesis:  
 \begin{equation}
0\leq f_{0} (x,v)\leq C_1\, e^{- \lambda |v|^2} \frac{1}{(1+|x|)^{\alpha}}   \label{dec}
 \end{equation}
with $\alpha > 1$, and $\lambda$,  $C_1$,  positive constants. Then there exists a solution to system (\ref{ch}) in $[0,T]$ 
and positive constants $C_2$ and $\mu$ such that  
 \begin{equation}
 0\leq f(x,v,t)\leq C_2 \, e^{- \mu |v|^2} \frac{1}{(1+|x|)^{\alpha}} \label{dec2}.
 \end{equation}
This solution is unique in the class of those satisfying (\ref{dec2}).
\label{3}
\end{theorem}

We remark that the condition
$\alpha >1$ is  assumed  to have a finite gravitational field at time $0$.
We will concentrate in the following on the case (with infinite mass) $1<\alpha<3$, since the cases $\alpha=3$
(a border case with infinite mass) and $\alpha>3$ are simpler.




\bigskip

The proof is analogous to those in Ref.s \cite{CCM18, CCM20}, once some fundamental estimates on the gravitational
and kinetic energies are established, which is the aim of the next section.

\section {Partial dynamics}

 A truncated dynamics, named also \textit{partial dynamics}, is introduced in order to deal with a system of finite mass and size, for which well known results of existence and uniqueness of the solution hold:
\begin{equation}
 f_{0}^N({x}, {v})=f_{0}({x}, {v})\chi_{\{|x|\leq N^\beta\}} (x)  \chi_{\{|v|\leq N\}} (v)  \label{B0}
\end{equation}
where $\chi_{\{\cdot\}}(\cdot)$ is the characteristic function of the set $\{\cdot\}$, and $f_{0}$ satisfies condition \eqref{dec}.
We introduce a velocity cutoff, $N$, and a spatial cutoff, $N^\beta$ (with $\beta>0$ to be fixed later), and we investigate the limit $N\to\infty$.

System (\ref{ch}) then becomes
\begin{equation}
\label{ch'} 
\begin{cases}
\dsp  \dot{X}^N(t)= V^N(t)\\
\dsp  \dot{V}^N(t)= G^N(X^N(t),t) \\
\dsp (X^N(0), V^N(0))=(x,v)  \\
\dsp f^N(X^N(t), V^N(t), t) = f_{0}^N(x,v),
 \end{cases}
\end{equation}
where
$$
G^N(x,t)=- \int_{\R^3 } \frac{x-y}{|x-y|^3} \ \rho^N(y,t) \, dy     
$$
$$
\rho^N(x,t)=\int_{\R^3} f^N(x,v,t) \, dv.
$$

Since $f_{0}^N$ has compact support, system (\ref{ch'})  admits a unique solution globally in time
(see for instance \cite{G}).
We want to show that, when $N\to \infty$, such solution converges to the unique solution of system (\ref{ch}).

\noindent In order to perform this limit, which is constructively obtained by means of an iterative procedure, 
we need to point out the dependence on the cutoff $N$ in the estimates of the kinetic energy,  the
gravitational energy, and the gravitational field.

We give a sketch of the proof.  The heart of the matter consists to control  the velocities of the  particles, and to this aim we need to bound the gravitational field produced by the whole system.
The \textit{a priori} bound of the maximum value of the gravitational field is useless for our scope, whereas a better bound,
useful for our technique, can be obtained if we consider the time integral of $|G^N|$.
In other words  we define the {\textit{maximal velocity}} of the particles in the partial dynamics as
 \begin{equation}
{\mathcal{V}}^N(t)=  \max\left\{ C_3,\sup_{s\in[0,t]}\sup_{(x,v)\in B_x \times B_v} 
|V^N(s)|\right\} \label{mv}
\end{equation}
where $B_x=B(0, N^\beta)$,  $B_v=B(0, N)$, are balls in $\mathbb{R}^3$ of center $0$ and radius $N^\beta$, $N$, respectively, and the constant $C_3>1$ is suitably chosen.
In the same way as in \cite{CCM18}  (Proposition 2.7) it can be proved that
\begin{equation*}
\int_0^t|G^N(X^N(s),s)|ds\leq C \left[{\mathcal{V}}^N(t)\right]^\tau \quad \quad \hbox{with}\quad \tau <\frac23,  
\end{equation*}
for any $t\in[0, T]$,
which permits to control the particles'  velocities in the partial dynamics, that is to establish the bound
${\mathcal{V}}^N(t) <C N$, 
 which is fundamental for the convergence
of the iterative method involving the quantities $|X^N(t)-X^{N+1}(t)|$ and $|V^N(t)-V^{N+1}(t)|$.

\medskip

The total energy of the truncated system is
\begin{equation}
\mathcal{E}^N(t) = \frac12 \int d x \int d v \, |v|^2  f^N( x, v,t) 
 - \frac12 \int dx \,
 \rho^N( x,t)\int dy \, 
\frac{ \rho^N ( y,t)}{  | x- y|}
\label{W_en}
 \end{equation}
which is decomposed into a kinetic and a potential energy.
Since we are referring to the partial dynamics,
the total energy is finite (with a dependence on $N$ which will be made explicit) and it is conserved in time,
$
\mathcal{E}^N(t) = \mathcal{E}^N(0).
$

 We recall a classical argument which allows to bound the potential energy by the square root of
 the kinetic energy. Then we obtain the explicit dependence on $N$ of the kinetic 
 energy, which permits to make the iterative method work.
 
 \begin{proposition}
 Under the hypothesis of Theorem \ref{3}, it results
 \begin{equation}
 \frac12 \int d x \int d v \, |v|^2  f^N( x, v,t) \le  C N^{\frac73\beta(3-\alpha)}.
 \end{equation}
 \end{proposition}
 \begin{proof}
 Let us put 
$$
 T^N(t)= \frac12 \int d x \int d v \, |v|^2  f^N( x, v,t)
 $$
 for the kinetic energy, and
 $$
 U^N(t)= - \frac12 \int dx \,
 \rho^N( x,t)\int dy \, 
\frac{ \rho^N ( y,t)}{  | x- y|}
$$
for the potential energy, in such a way that the total energy is 
\begin{equation}
\label{en^N}
\mathcal{E}^N(t)=T^N(t)+U^N(t)=\mathcal{E}^N(0).
\end{equation}

\noindent Notice that, for any $a>0$, it is
\begin{equation*}
\begin{split}
\rho^N(x,t)=&\int  f^N(x,v,t)\,dv\leq \int_{|v|\leq a}  f^N(x,v,t)\,dv\,+\\&\frac{1}{a^2}\int_{|v|>a}\ |v|^2 f^N(x,v,t)\,dv
\end{split}
\end{equation*}
so that, by (\ref{f0}), we have 
\begin{equation}
\rho^N(x,t)\leq C a^3+\frac{1}{a^2}k^N(x,t)
\end{equation}
where
 $$
 k^N(x,t):= \int  |v|^2 f^N(x,v,t)\,dv.
 $$
 By minimizing over $a$ and  taking the power $\frac53$ of both members we get 
\begin{equation*}
\rho^N(x,t)^{\frac53} \leq C k^N(x,t).
\end{equation*}
We evaluate the following integral appearing in the definition of $U^N(t)$,
\begin{equation}
\begin{split}
\int dy \, 
\frac{ \rho^N ( y,t)}{  | x- y|} &= \int_{|x-y|\le R} dy \, 
\frac{ \rho^N ( y,t)}{  | x- y|} + \int_{|x-y|>R} dy \, 
\frac{ \rho^N ( y,t)}{  | x- y|}\\
 &:= A^N(x,t)+B^N(x,t).
\end{split}
\end{equation}
It results
\begin{equation}
B^N(x,t) \le \frac1R \int  dy \, 
\rho^N ( y,t) \le \frac{C}{R} N^{\beta(3-\alpha)}
\end{equation}
by  \eqref{dec}, the spatial cutoff $N^\beta$ introduced in (\ref{B0}), and the conservation of the total mass.
For $A^N(x,t)$ we have
\begin{equation}
\begin{split}
A^N(x,t) &\le \left(\int dy \, 
\rho^N ( y,t)^\frac53   \right)^\frac35
\left( \int_{|x-y|\le R} dy \, \frac{1}{  | x- y|^\frac52}  \right)^\frac25 \\
&\le C R^\frac15    \left(\int dy \, 
k^N(y,t)  \right)^\frac35 \le C R^\frac15 \left[ T^N(t)  \right]^\frac35.
\end{split}
\end{equation}
Therefore, by definition of $U^N(t)$ and again by  \eqref{dec}, the spatial cutoff $N^\beta$,  and the conservation of the total mass,
\begin{equation}
\begin{split}
|U^N(t)|&\le  C\left\{\frac{N^{\beta(3-\alpha)}}{R}  +  R^\frac15  \left[ T^N(t)  \right]^\frac35   \right\} \int dx \,
 \rho^N( x,t)  \\
 &\le  C N^{\beta(3-\alpha)} \left\{\frac{N^{\beta(3-\alpha)}}{R}  +  R^\frac15  \left[ T^N(t)  \right]^\frac35   \right\}.
 \end{split}
\end{equation}
Hence by \eqref{en^N},
\begin{equation}
\begin{split}
T^N(t) &= \mathcal{E}^N(0)  -U^N(t) \\
 &\le \mathcal{E}^N(0) + C N^{\beta(3-\alpha)} \left\{\frac{N^{\beta(3-\alpha)}}{R}  +  R^\frac15  \left[ T^N(t)  \right]^\frac35   \right\}.
 \end{split}
\end{equation}
If we weren't interested in letting $N\to \infty$ (that is,  keeping  $N$ constant), we could optimize  this expression by taking
$R=Const. \left[ T^N(t)  \right]^{-1/2}$, obtaining $|U^N(t)|\le Const. \left[ T^N(t)  \right]^{1/2}$
and then $T^N(t)$ finite.
In the present case, since we are going to take the limit $N\to \infty$, it is appropriate to make explicit the
dependence on $N$, by choosing
$$
R=N^{\frac56\beta(3-\alpha)} \left[ T^N(t)  \right]^{-1/2},
$$ 
from which, recalling that  by \eqref{dec} and the spatial cutoff $N^\beta$ it is  $\mathcal{E}^N(0)\le C N^{\beta(3-\alpha)}$,  we have
$$
T^N(t) \le  C N^{\frac73\beta(3-\alpha)}.
$$

 \end{proof}
 
 With the previous estimate on the kinetic energy the rest of the proof follows the same steps of Ref.
 \cite{CCM18} (which was strongly based on the positivity of the total energy).  In fact in \cite{CCM18} it is proved that a mollified version of the energy of a ball of radius
 $R^N(t)$ (the maximum displacement of the particles) is bounded by $C N^{1-\eps}$, with
 $\frac{1}{15}<\eps<1$.
 This condition is fulfilled  in the present situation by the kinetic energy (which is the relevant 
 quantity in the proof) if
 $$
 0<\frac73\beta(3-\alpha)< \frac{14}{15},
 $$
 that is by choosing any $\beta$ satisfying
 $$
 0<\beta< \frac{2}{5(3-\alpha)},
 $$
 (recalling that $1<\alpha<3$).  Hence from this point the rest of the proof is analogous to the one in Ref. \cite{CCM18}, and we do not repeat it.

\bigskip

\bigskip

\bigskip

\bigskip

\textbf{Acknowledgments}.
Work performed under the auspices of 
GNFM-INDAM and the Italian Ministry of the University (MUR).

\end{document}